\documentstyle[11pt,amssymb,amsmath,amscd,amsbsy,amsfonts,theorem,hyperref]{article}

\input xy
\xyoption{all}

\newcommand{\zz}{{\Bbb Z}}

\newcommand{\nn}{\Bbb N}
\newcommand{\cc}{\Bbb C}

\newcommand{\pp}{{\Bbb P}}

\newcommand{\iis}{{\mathbf{i}}}
\newcommand{\lls}{{\mathbf{l}}}
\newcommand{\ddim}{\operatorname{dim}}
\newcommand{\ddeg}{\operatorname{deg}}
\newcommand{\kker}{\operatorname{Ker}}

\newcommand{\op}[1]{\operatorname{#1}}

\newcommand{\ffi}{\varphi}

\newcommand{\eps}{\varepsilon}

\newcommand{\la}{\langle}
\newcommand{\ra}{\rangle}
\newcommand{\row}{\rightarrow}

\newcommand{\low}{\leftarrow}
\newcommand{\lrow}{\longrightarrow}
\renewcommand{\leq}{\leqslant}
\renewcommand{\geq}{\geqslant}

\newcommand{\nichego}[1]{}

\newcommand{\ov}[1]{\overline{#1}}

\newcommand{\wt}[1]{\widetilde{#1}}

\newcommand{\smk}{{\mathbf{Sm_k}}}

\newcommand{\laz}{{\Bbb L}}
\newcommand{\co}{{\cal O}}

\newcommand{\cl}{{\cal L}}
\newcommand{\cm}{{\cal M}}

\newcommand{\cv}{{\cal V}}
\newcommand{\cn}{{\cal N}}

\newcommand{\che}[2]{c^{\Omega}(#1;#2)}
\newcommand{\cheCH}[2]{c(#1;#2)}

\newcommand{\Qed}{\hfill$\square$\smallskip}
\newenvironment{proof}{\noindent{\it Proof}:}{\vskip 5mm}

\newtheorem{prop}{Proposition}[section]{\bf}{\it}
\newtheorem{thm}[prop]{Theorem}{\bf}{\it}
\newtheorem{lem}[prop]{Lemma}{\bf}{\it}
{\bf}{\it}
\newtheorem{defi}[prop]{Definition}{\bf}{\it}
{\bf}{\it}
{\bf}{\it}
{\bf}{\it}
\newtheorem{rem}[prop]{Remark}{\bf}{}

{\bf}{\it}
\newtheorem{claim}[prop]{Claim}{\bf}{\it}
\newtheorem{cor}[prop]{Corollary}{\bf}{\it}
{\bf}{\it}

\begin{document}

\title{Symmetric operations for all primes and Steenrod operations in
Algebraic Cobordism}
\author{Alexander Vishik\footnote{School of Mathematical Sciences, University
of Nottingham}}
\date{}

\maketitle

\begin{abstract}
In this article we construct Symmetric operations for all primes (previously known only for $p=2$).
These unstable operations are more subtle than the Landweber-Novikov operations, and encode all
$p$-primary divisibilities of characteristic numbers. Thus, taken together (for all primes) they
plug the gap left by the Hurewitz map $\laz\hookrightarrow\zz[b_1,b_2,\ldots]$, providing an
important structure on Algebraic Cobordism. Applications include: questions of rationality of Chow
group elements - see \cite{GPQCG}, and the structure of the Algebraic Cobordism - see \cite{ACMLR}.
We also construct Steenrod operations of T.tom Dieck-style in Algebraic Cobordism. These unstable
multiplicative operations are more canonical and subtle than Quillen-style operations, and complement
the latter.
\end{abstract}

\tableofcontents

\section{Introduction}
\label{Intro}

The standard tool to distinguish and classify algebraic varieties is the
use of various cohomology theories. The prominent place among them is occupied
by Chow groups $\op{CH}^*$ and $K_0$. Many results and conjectures in Algebraic
Geometry are formulated in terms of these theories. But as was shown by M.Levine
and F.Morel, these two theories are just small "faces" of much larger Algebraic Cobordism
theory $\Omega^*$, which is an algebro-geometric analogue of the complex-oriented
cobordism $MU^*$ in topology. This theory is rich because of the abundance of
cohomological operations on it. The stable ones among them are provided by the
Landweber-Novikov operations (see \cite[Example 4.1.25]{LM}). These operations
permit to prove (as well as to formulate!) many interesting results on Algebraic Cobordism,
on Chow groups, and on $K_0$. But for some sorts of questions (related to torsion effects)
these operations are not subtle enough. The remedy is the use of unstable operations.
As the first example of such operations on $\Omega^*$, the {\it Symmetric operations} for $p=2$
were constructed in \cite{so1} and \cite{so2}. In \cite{GPQCG} these operations were
applied to the question of rationality of the Chow group elements, and they provide the
only known way to deal with the 2-torsion there. More generally, it was observed that
these operations control all 2-primary divisibilities of characteristic numbers, and
thus plug 2-adically the gap between $\laz=\Omega^*(\op{Spec}(k))=\pi_*(MU)$
and $\zz[b_1,b_2,\ldots]=H_*(MU)$ left by the Hurewitz map. In Topology an analogous
observation is (implicitly) contained in the beatiful work \cite{Qu71} of D.Quillen,
where some traces of the topological counterparts of the mentioned operations are used.
As soon as one realizes that Symmetric operations (for $p=2$) of \cite{so2} (constructed with a
completely different purpose in mind) are analogous to operations used by D.Quillen to
describe $MU^*(pt)$, the natural desire appears to control the mentioned "gap" $p$-adically
for odd primes as well, taking into account that in Topology D.Quillen deals with such
primes in the same way as with $p=2$.
The needed Total Symmetric operation for the given $p$ must be the "negative part" of the
(Quillen's type) Total Steenrod operation mod $p$ divided by "formal $p$". The problem
though is to divide canonically.
The obstacle here is that in our context we are deprived
of the standard topological tools, and have to work not with spectra, but with cohomology theories
themselves. The case $p=2$ was done by an explicit geometric construction (using $Hilb_2$), and
it is not clear at all how to generalize it. Besides, there existed no general methods to
construct unstable operations in algebraic geometry, and aside from classical {\it Adams operations}
in $K_0$ and mentioned {\it Symmetric operations} for $p=2$ no examples were known.
After several years of attempts to construct {\it Symmetric operations} for arbitrary $p$,
the author finally found an approach which permitted to describe all unstable additive operations
in $\Omega^*$ and all theories obtained from it by change of coefficients - see \cite{SU}.
In this approach, to construct an operation, one only needs to define it on the
powers of the $\pp^{\infty}$, and check that some compatibility rules are satisfied -
see Theorem \ref{MAIN}. One of the consequences is that we can "divide" operations canonically, since
$(\pp^{\infty})^{\times r}$ is {\it cellular}, and $\laz$ is an integral domain.
This way, we obtain {\it Symmetric operations} for all primes - see Theorem \ref{SOp}.

We also construct Steenrod operations of T.tom Dieck-style in Algebraic Cobordism. These
operations are "more canonical" than Quillen-style Steenrod operations (in contrast to the latter,
they depend on $p$ only). And while the Quillen-style Total Steenrod operation is
just a specialization of the Total Landweber-Novikov operation, the one of the T.tom Dieck-style is
an unstable multiplicative operation, so a much more subtle object. Our construction uses
some derivatives of the Theorem \ref{MAIN} and a nice Theorem \ref{G} describing the invariants of
the continuous group action on a power series ring.

This text is organized as follows: In Section \ref{AC} we provide basic facts about Algebraic Cobordism
and other {\it Generalized oriented cohomology theories}. In Section \ref{Op} we discuss
cohomological operations between theories and introduce the notion of a {\it theory of rational type}.
This class of theories contains $\Omega^*$, $\op{CH}^*$ and $K_0$ and permits a complete description
of the set of additive cohomological operations (obtained in \cite{SU}). In Section \ref{2St}
we compare Steenrod operations of Quillen and T.tom Dieck-styles in Cobordisms. In Section \ref{cont}
the continuous group action on the power series ring is studied. In Section \ref{TtD} the results
of Sections \ref{Op} and \ref{TtD} are applied to produce the Steenrod operations of T.tom Dieck-style. And, finally, in Section \ref{Sp} we construct {\it Symmetric operations} for all primes $p$, and deduce some
properties of them.

{\bf Acknowledgements:}
I want to thank P.Brosnan, S.Gille, O.Haution, A.Lazarev, M.Levine, F.Morel, I.Panin, M.Rost,
A.Smirnov, V.Voevodsky and other people for very useful conversations.
The support of EPSRC Responsive Mode grant EP/G032556/1 is gratefully acknowledged.
I'm grateful to the Referee for various suggestions and remarks which improved the exposition.

\section{Algebraic Cobordism of Levine-Morel}
\label{AC}

\subsection{Generalized oriented cohomology theories}

In this article, $k$ will denote a field of
characteristic zero, and $\smk$ the category of smooth quasi-projective
variaties over $k$. The notion of generalized oriented cohomology
theory in Algebraic Geometry is borrowed from Topology (D.Quillen -
\cite{Qu71}) with some variations.

Such a theory assigns to each smooth quasi-projective variety $X$
a (commutative and, possibly, graded) ring $A^*(X)$, together with the
structure of pull-backs
$f^*:A^*(Y)\row A^*(X)$ for all maps $f:X\row Y$,
and the structure of push-forwards
$f_*:A_*(X)\row A_*(Y)$ for all projective maps of constant relative dimension
(where $A_*(X):=A^{\ddim(X)-*}(X)$ for equi-dimensional $X$).
These must satisfy certain set of axioms.
We will use the definition of \cite[Definition 2.1]{SU} which is the
definition of M.Levine-F.Morel (\cite[Definition 1.1.2]{LM}) plus the
localization (excision) axiom $(EXCI)$.
So, everywhere below under ``cohomology theory'' we will mean a
theory satisfying \cite[Definition 2.1]{SU}.

In \cite{LM} M.Levine and F.Morel constructed a universal
generalized oriented cohomology theory $\Omega^*$ on $\smk$
called Algebraic Cobordism (see \cite{LP} for an alternative definition).
For a smooth quasi-projective $X$, the ring $\Omega^*(X)$ is additively
generated by the classes $[V\stackrel{v}{\row}X]$ of projective maps with
smooth $V$, modulo certain relations.
This theory has a unique morphism of theories $\Omega^*\row A^*$
to any other theory $A^*$. If $k$ has a complex embedding, there is a natural
{\it topological realization functor} $\Omega^*(X)\row MU^{2*}(X(\cc))$ which is an isomorphism
for $X=\op{Spec}(k)$. In the case of Chow groups, the natural morphism
$\Omega^*\row\op{CH}^*$ is surjective, and moreover,
$\op{CH}^*(X)=\Omega^*(X)\otimes_{\laz}\zz$ (\cite[Theorem 1.2.18]{LM}).
The same is true about $K_0$ by \cite[Theorem 1.2.19]{LM}. Thus, Chow groups and $K_0$
can be reconstructed out of $\Omega^*$.

Since we will not work with the axioms, we will not reproduce them here,
but we mention that any theory which is obtained from Algebraic Cobordism
of Levine-Morel by change of coefficients: $A^*=\Omega^*\otimes_{\laz}A$
is a theory in our sense. These are the {\it free theories} in the sense
of M.Levine-F.Morel (\cite[Remark 2.4.14]{LM}), and are exactly the
{\it theories of rational type} of \cite{SU} (see \cite[Proposition 4.9]{SU}).
In particular, the theories $\Omega^*$, $\op{CH}^*$ and $K_0$ are such.

\subsection{Formal group law}
\label{FGL}

Any theory in the above sense has Chern classes: a set of elements
$c^A_i(\cv)\in A^i(X)$ assigned to each vector bundle $\cv$ on $X$,
which satisfy the Cartan formula, and in the case of a linear bundle $\cl$,
$c^A_1(\cl)=s^*s_*(1)$, where $s:X\row\cl$ is a zero section (see \cite{PS}
or \cite{P})).
By \cite[Theorem 2.3.13]{LM}, any theory $A^*$ as above satisfies the axiom:
\begin{itemize}
\item[$(DIM)$ ] For any line bundles $\cl_1,\ldots,\cl_n$ on a
smooth $X$ of dimension $<n$,
one has: \\
$c^A_1(\cl_1)\cdot\ldots\cdot c^A_1(\cl_n)=0\in A_*(X)$.
\end{itemize}
Thus, any power series in Chern classes can be evaluated on any
element of $A_*(X)$.

To any generalized oriented cohomology theory $A^*$ one can associate
the Formal Group Law $(A,F_A)$. Here $A$ is
the coefficient ring of $A^*$, and
$$
F_A(x,y)=Segre^*(t)\in A[[x,y]]=A^*(\pp^{\infty}\times\pp^{\infty}),
$$
where
$\pp^{\infty}\times\pp^{\infty}\stackrel{Segre}{\lrow}\pp^{\infty}$
is the Segre embedding,
and $x,y,t$ are the $1$-st Chern classes of $\co(1)$ of the respective
copies of $\pp^{\infty}$ (recall, that due to the {\it projective bundle axiom},
$A^*(\pp^{\infty})=\varprojlim A^*(\pp^n)=A[[t]]$).
Denoting the coefficients of $F_A$ as $a^A_{i,j}$, we get:
$$
F_A(x,y)=\sum_{i,j}a^A_{i,j}\cdot x^i\cdot y^j.
$$
The formal group law describes how to compute the $1$-st Chern class
of a tenzor product of two
line bundles in terms of the $1$-st Chern classes of factors:
$$
c^A_1(\cl\otimes\cm)=F_A(c^A_1(\cl),c^A_1(\cm)).
$$
The universal formal group law $(\laz,F_U)$ has canonical morphism
to any other formal group law,
in particular, to $(A,F_A)$. M.Levine and F.Morel have shown that,
in the case of algebraic cobordism,
the respective map is an isomorphism - see \cite[Theorem 1.2.7]{LM}.
In particular, for any field $k$,
$\Omega^*(k)=\laz$ - the Lazard ring. As an abstract graded ring,
$\laz\cong\zz[x_1,x_2,\ldots]$, where $\ddeg(x_i)=i$.

\section{Operations}
\label{Op}

\subsection{General facts}
To study cohomology theories effectively one needs a reasonable notion
of ``morphisms'' between them. If we restrict ourselves to maps respecting
both pull-backs and push-forwards, then there will be not many of such
(for example, there will be only one map $\Omega^*\row A^*$, for any $A^*$).
So, we have to permit more flexibility. The experience of Topology suggests
that the right thing is to require that only pull-backs are respected.

\begin{defi}
\label{operation}
Let $A^*$ and $B^*$ be theories in the above sense.
An operation
$G:A^n\row B^m$ is a morphism of contravariant functors of sets pointed by $0$
(in other words, a transformation commuting with the pull-backs, and
sending zero to zero).
An operation is called additive, if it is a homomorphism of abelian group.
An operation $G:A^*\row B^*$ is called multiplicative, if it is a
homomorphism of rings.
\end{defi}

To each multiplicative operation one can assign certain power series - the
{\it inverse Todd genus}
$\gamma_G=b_0\cdot x+b_1\cdot x^2+b_2\cdot x^3+\ldots\in B[[x]]$, where,
for $x^A=c^A_1(\co(1))$, $x^B=c^B_1(\co(1))$, one has:
$G(x^A)=\gamma_G(x^B)\in B[[x^B]]=B(\pp^{\infty})$.
Also, we have $\ffi_G:A\row B$ - the homomorphism of coefficient rings.
The pair $(\ffi_G,\gamma_G)$ is a morphism of formal group laws:
$(A,F_A)\lrow (B,F_B)$.
In other words,
$$
\ffi_G(F_A)(\gamma_G(u),\gamma_G(v))=\gamma_G(F_B(u,v)).
$$
Of course, the composition of multiplicative operations corresponds to the
composition of morphisms of formal group laws:
$$
(\ffi_{H\circ G}\,,\,\gamma_{H\circ G}(x))=
(\ffi_H\circ\ffi_G\,,\,\ffi_H(\gamma_G)(\gamma_H(x))).
$$
In the case of $A^*=\Omega^*$, and $b_0$ invertible in $B$, the homomorphism
$\ffi_G$ is completely determined by $\gamma_G$.
Namely, $\laz$ is generated as a ring by universal coefficients
$a^{\Omega}_{i,j}$,
and $\ffi_G(a^{\Omega}_{i,j})$ is the respective coefficient of the
formal group law
$F^{\gamma_G}_B(x,y)=\gamma_G(F_B(\gamma_G^{-1}(x),\gamma_G^{-1}(y)))$.
Moreover, from the reorientation procedure of I.Panin-A.Smirnov (see \cite{PS},\cite{P},\cite{Sm1})
and universality of $\Omega^*$ of M.Levine-F.Morel
(see \cite[Theorem 1.2.6]{LM}) one obtains:
\begin{thm} {\rm (Panin-Smirnov+Levine-Morel)}
\label{PSLM}
If $b_0$ is invertible in $B$, then for each
$\gamma=b_0x+b_1x^2+b_2x^3+\ldots\in B[[x]]$, there exists unique
multiplicative operation $G:\Omega^*\row B^*$ with $\gamma_G=\gamma$.
\end{thm}

One can easily introduce the notion of a {\it stable operation} -
see (for example) \cite[Definition 3.4]{SU}.
As in Topology, these are operations commuting with some sort of suspension.
To define the latter we need to consider theories on pairs $(X,U)$ where
$U$ is an open subvariety of a smooth variety $X$. Fortunately, every theory in
our sense naturally extends to pairs by the rule: $A^*((X,U)):=\kker(A^*(X)\row A^*(U))$,
where we have to admit the non-unital rings into the game.
And our suspension is just the smash-product with $(\pp^1,\pp^1\backslash\{0\})$ - see \cite{SU}.
An operation extends naturally to pairs as well, as long as it is "pointed" (sends $0$ to $0$).
Hence, we can talk about stability.
We should mention the following simple result
(see, for example, \cite[Proposition 3.9]{SU}):

\begin{prop}
\label{multstable}
Let $G:A^*\row B^*$ be a multiplicative operation with
$\gamma_G=b_0x+b_1x^2+\ldots$.
Then $G$ is stable if and only if $b_0=1$.
\end{prop}

The most important, and, in a sense, universal example of a stable multiplicative operation is
provided by the {\it Total Landweber-Novikov operation} (see \cite[Example 4.1.25]{LM}):
$$
S^{Tot}_{L-N}:\Omega^*\row\Omega^*[b_1,b_2,\ldots],
$$
with $\gamma_{S^{Tot}_{L-N}}(x)=x+b_1x^2+b_2x^3+\ldots$.
Individual Landweber-Novikov operations are coefficients of the total one at particular $b$-monomials.

Any stable multiplicative operation $G:\Omega^*\row B^*$ is a specialization of $S^{Tot}_{L-N}$.
That is, $G=\eta\circ S^{Tot}_{L-N}$, where $\eta:\Omega^*[b_1,b_2,\ldots]\row B^*$ is
a morphism of theories sending $b_i$'s to the coefficients of $\gamma_G$.
Similarly, any multiplicative operation as in Theorem \ref{PSLM} (i.e., with invertible $b_0$),
is a {\it generalized specialization} of $S^{Tot}_{L-N}$. In other words, it is the composition of
the reparametrization $\cdot b_0^{codim}:\Omega^*\row\Omega^*$ and the specialization as above corresponding
to $\frac{\gamma_G}{b_0}$ - see \cite[Section 3]{ACMLR}.

\subsection{Case of a theory of rational type}

Proposition \ref{PSLM} provides an effective tool in
constructing stable operations from Algebraic Cobordism theory
elsewhere. But in many situations one has to work with operations
where $b_0$ is not invertible in the coefficient ring of the target theory
$B^*$. The needed tools are provided by the results of \cite{SU} on {\it theories of
rational type}.

These theories are defined in \cite[Definition 4.1]{SU}, but for us it
will be important that these are exactly the {\it free theories} of
M.Levine-F.Morel, that is, the theories obtained from $\Omega^*$ by
change of coefficients: $A^*=\Omega^*\otimes_{\laz}A$.

The principal result on multiplicative operations here is:

\begin{thm} {\rm (\cite[Theorem 6.8]{SU})}
\label{multFGL}
Let $A^*$ be theory of rational type, and $B^*$ be any theory
in the above sense.
The assignment $G\leftrightarrow (\ffi_G,\gamma_G)$ defines a
$1$-to-$1$ correspondence between the set of multiplicative operations
$A^*\stackrel{G}{\row}B^*$ and the set of homomorphisms $(A,F_A)\row (B,F_B)$
of the respective formal group laws.
\end{thm}

This immediately gives an extension of Theorem \ref{PSLM}.

\begin{thm} {\rm (\cite[Theorem 6.9]{SU})}
\label{neobrB0}
Let $B^*$ be any theory in the above sense, and
$b_0\in B$ be not a zero-divisor.
Let $\gamma=b_0x+b_1x^2+b_2x^3+\ldots\in B[[x]]$. Then there exists a
multiplicative operation
$\Omega^*\stackrel{G}{\row}B^*$ with $\gamma_G=\gamma$ if and only if
the shifted FGL
$F_B^{\gamma}\in B[b_0^{-1}][[x,y]]$ has coefficients in $B$ (that is,
has no denominators).
In this case, such an operation is unique.
\end{thm}

This result will enable us to construct the Steenrod operations of
T.tom Dieck style below.

The methods of \cite{SU} and \cite{PO} permit to work with non-multiplicative operations
as well. The main result (see also \cite[Theorem 5.1]{SU} for the additive version)
which implies all the rest is the following:

\begin{thm} {\rm (\cite[Theorem 5.1]{PO})}
\label{MAIN}
Let $A^*$ be a {\it theory of rational type}, and
$B^*$ be any theory in the above sense.
Then the set of operation $A^n\stackrel{G}{\row}B^*$ on $\smk$ is
identified with the set of transformations
$$
A^n((\pp^{\infty})^{\times l})\stackrel{G}{\row}B^*((\pp^{\infty})^{\times l}),\,\,\text{for}\,\,
l\in\zz_{\geq 0}
$$
commuting with the pull-backs for:
\begin{itemize}
\item[$(i)$ ] the action of the symmetric group ${\frak{S}}_l$;
\item[$(ii)$ ] the partial diagonals;
\item[$(iii)$ ] the partial Segre embeddings;
\item[$(iv)$ ] $(\op{Spec}(k)\hookrightarrow\pp^{\infty})\times(\pp^{\infty})^{\times r}$,
$\forall r$;
\item[$(v)$ ] the partial projections.
\end{itemize}
\end{thm}

In Topology an analogous result was obtained by T.Kashiwabara - see \cite[Theorem 4.2]{Kash}.

Under an {\it additive subtheory} $C^*$ of a theory $B^*$ we will mean an
assignment $X\mapsto C^*(X)$, where $C^*(X)\subset B^*(X)$ is an additive subgroup,
closed under pull-backs and push-forwards, and satisfying the axioms
$(A1),(A2),(PB),(EH),(EXCI)$ of \cite[Definition 2.1]{SU} (thus, only the
axioms $(D1),(D2)$ are substituted by the notion of an additive subobject of such).

Theorem \ref{MAIN} immediately implies:

\begin{prop}
\label{Qdiv}
Let $A^*$ be a {\it theory of rational type},
$B_i^*$, $i=1,2$ be any theories in our sense,
and $C_i^*\subset B_i^*$, $i=1,2$ be additive subtheories.
Let $Q:C_1^*\row C_2^*$ be an additive operation, such that $Q|_{(\pp^{\infty})^{\times l}}$ is injective,
for all $l\in\zz_{\geq 0}$,
and $G:A^n\row C_2^*$ be an operation
such that $image(G|_{(\pp^{\infty})^{\times l}})\subset image(Q|_{(\pp^{\infty})^{\times l}})$,
for all $l\in\zz_{\geq 0}$.
Then there exists unique operation
$H:A^n\row C_1^*$ such that $G=Q\circ H$.
\end{prop}

\begin{proof}
By our condition, the transformation:
$$
A^n((\pp^{\infty})^{\times l})\stackrel{G}{\row}C_2^*((\pp^{\infty})^{\times l}),\,\,\text{for}\,\,
l\in\zz_{\geq 0}
$$
corresponding to $G$ can be written in a unique way as the composition of some transformation
$$
A^n((\pp^{\infty})^{\times l})\stackrel{H}{\row}C_1^*((\pp^{\infty})^{\times l}),\,\,\text{for}\,\,
l\in\zz_{\geq 0}
$$
and the operation $Q$.
The fact that $H$-transformations will commute with all the pull-backs prescribed in the
Theorem \ref{MAIN} follows from the respective property for $G$ and $Q$ together with the injectivity
of $Q$. Hence, it can be extended to a unique
operation $H:A^n\row C_1^*$ (a'priori we get an operation $H:A^n\row B_1^*$, but
it lands in $C_1^*$, because it is so on $(\pp^{\infty})^{\times l}$, for all $l$ - can be seen from
the proof of Theorem \ref{MAIN}).
The fact that $G=Q\circ H$ is clear from the same Theorem \ref{MAIN}
\Qed
\end{proof}

\begin{cor}
\label{div}
Let $A^*$ be a {\it theory of rational type}, and
$B^*$ be any theory in the above sense.
\begin{itemize}
\item[$(1)$ ] Let $B^*=C^*\oplus D^*$ be an additive decomposition, and
$b\in B$ be such an element, that the composition
$m_b:C\stackrel{i_C}{\hookrightarrow}B
\stackrel{\cdot b}{\row}B\stackrel{\pi_C}{\twoheadrightarrow}C$ is injective.
Let $G:A^n\row B^*$ be an operation
such that $image(\pi_C\circ G|_{(\pp^{\infty})^{\times l}})\subset
image(m_b|_{(\pp^{\infty})^{\times l}})$, for all $l\in\zz_{\geq 0}$. Then there exists unique operation
$H:A^n\row C^*\hookrightarrow B^*$ such that
$$
(G-b\cdot H):A^n\row D^*\hookrightarrow B^*.
$$
\item[$(2)$ ]
Let $G:A^n\row B^*$ be an operation, and $b\in B$ be not a zero-divisor
such that the $image(G|_{(\pp^{\infty})^{\times l}})$, for all $l\in\zz_{\geq 0}$
is divisible by $b$. Then there exists unique operation
$H:A^n\row B^*$ such that $G=b\cdot H$.
\end{itemize}
\end{cor}

\begin{proof}
1) Apply Proposition \ref{Qdiv} with $B_i=B$, $C_i=C$, for $i=1,2$,
$Q$ - the composition:
$C^*\hookrightarrow B^*\stackrel{\cdot b}{\row}B^*\twoheadrightarrow C^*$
and $\pi_C\circ G:A^n\row C^*$.
We obtain a unique operation $H:A^n\row C^*$ such that
$\pi_C\circ G=Q\circ H$. Or, in other words, $(G-b\cdot H):A^n\row D^*\hookrightarrow B^*$.

2) Take $C^*=B^*$ and $D^*=0$ in (1).
\Qed
\end{proof}

\section{Two types of Steenrod operations on Cobordisms}
\label{2St}

In Topology, all additive operations on singular homology modulo $p$
are generated by the Steenrod operations (and Bockstein).
These can be organized into a multiplicative Total Steenrod
operation, which can be constructed as follows.
Denoting as $H^*_{\zz/p}$ - the $\zz/p$-equivariant singular
homology, one gets a natural map
$$
H^m(X,\zz/p)\row H^{pm}_{\zz/p}(X^{\times p},\zz/p)\row H^{pm}_{\zz/p}(X,\zz/p),
$$
where the last arrow is induced by the diagonal embedding
$X\stackrel{\Delta}{\row}X^{\times p}$. And since the $\zz/p$-action
on $X$ is trivial, the target group can be identified with
$H^{pm}(X\times\op{B}\zz/p,\zz/p)$, which is the component of degree $pm$ of
$H^*(X,\zz/p)\otimes_{\zz/p}H^*(\op{B}\zz/p,\zz/p)=H^*(X,\zz/p)[[t]][y]/(y^2-c)$,
where $deg(y)=1$, $deg(t)=2$, and $c=0$, for $p>2$, and $c=t$, for $p=2$.
One obtains a multiplicative operation
$$
Sq:H^*(X,\zz/p)\row H^*(X,\zz/p)[[t]][y]/(y^2-c),
$$
whose only non-trivial components at monomials in $t$ are
of degrees divisible by $(p-1)$. These are the {\it individual Steenrod
operations}.
P.Brosnan has shown that the above construction goes through in the
algebro-geometric context (for Chow groups modulo $p$) as well - see \cite{Br}
(in a more general case of motivic cohomology these operations were produced
previously by V.Voevodsky by a different construction - see \cite{VoOP}).

Steenrod operations on singular cohomology modulo prime can be
extended to the theory of complex-oriented cobordisms $MU^*$
in (at least) two ways.

The first construction due to T.tom Dieck (\cite{tD}) uses the
same geometric approach with the $\zz/p$-equivariant cohomology replaced
by the $\zz/p$-equivariant cobordism, and gives the multiplicative operation:
$$
Sq:MU^*(X)\row MU^*(X\times\op{B}\zz/p)=
MU^*(X)[[t]]/(p\cdot_{MU}t)\row
MU^*(X)[[t]]/(\textstyle\frac{p\cdot_{MU}t}{t}).
$$
Note, in particular, that one gets a completely canonical operation
depending on $p$ only.

The second construction due to D.Quillen (\cite{Qu71}) is based
on the universal property of complex-oriented cobordism which implies that
any power series $\gamma=b_0x+b_1x^2+b_2x^3+\ldots$ with $b_0\in B$ invertible
corresponds to a unique multiplicative operation $G:MU^*\row B^*$. It remains
to specify $B^*$ and $\gamma$. One chooses representatives
$\{i_j,\,0<j<p\}$ of all non-zero cosets modulo $p$, and defines
$\gamma=x\prod_{j=1}^{p-1}(x+_{MU}i_j\cdot_{MU}t)\in\laz[[t]][[x]]$.
This gives a multiplicative operation:
$$
St(\ov{i}):MU^*\row MU^*[\iis^{-1}][[t]][t^{-1}],
$$
where $\iis:=\prod_{j=1}^{p-1}i_j$.
Note, that this operation depends on the choice of coset representatives
(of course, one can take $i_j=j$, for $j=1,\ldots,p-1$, but for general $p$,
such a choice will be about as good as any other).
As was shown by D.Quillen, his operation agrees with the one of T.tom Dieck.
Namely, there is the following commutative diagram:
$$
\xymatrix @-0.2pc{
MU^* \ar @{->}[r]^(0.3){St(\ov{i})} \ar @{->}[d]_(0.5){Sq}&
MU^*[\iis^{-1}][[t]][t^{-1}] \ar @{->}[d]^(0.5){}\\
MU^*[[t]]/(\frac{p\cdot_{MU}t}{t}) \ar @{->}[r]_(0.5){} &
MU^*[[t]][t^{-1}]/(p\cdot_{MU}t).
}
$$

The version of D.Quillen can be easily extended to the Algebraic Cobordism
of M.Levine-F.Morel using the universality of $\Omega^*$ (\cite[Theorem 1.2.6]{LM}). So, one gets
a multiplicative operation
$$
St(\ov{i}):\Omega^*\row \Omega^*[\iis^{-1}][[t]][t^{-1}],
$$
with the same $\gamma$ as above.
The situation with the version of T.tom Dieck is more delicate.
It is easy to define the $\zz/p$-equivariant Algebraic Cobordism,
but the problems appear when one tries to show that the natural map
$\Omega^m(X)\row\Omega_{\zz/p}^{pm}(X^{\times p})$ is well-defined.
The reason is that the defining relations in the Algebraic Cobordism
theory are more complicated than in the complex-oriented cobordism.
Namely, aside from the usual {\it elementary cobordism relations}
one has also the {\it double point relations} - see \cite{LP}.
It is rather easy to show that the elementary cobordism relations
are respected by our map, but the author was unable to do the
double point case. And although the author succeeded for $p=2$, he
had to employ the {\it Symmetric operations} modulo $2$.
Until now these operations were unavailable for $p>2$, and one
of the principal aims of the current article is to construct them.
So, we have to use a different approach. Fortunately, the methods
of \cite{SU} give us all the necessary tools.

\section{Continuous group action on a power series ring}
\label{cont}

To deal with the T.tom Dieck-style Steenrod operations we will need to
compute the invariants of the continuous action of a finite group on a
power series ring.

The following Lemma is the key to such a description.

\begin{lem} \label{del}
Let $R=\lim\limits_{\low}^{}R_n$ be a commutative ring, and $t_1,\ldots,t_r\in R$ be
such elements that $R$ is complete with respect to $t_i$
(that is, $t_i^n=0\in R_n$).
Suppose $t_i$ and $(t_i-t_j)$, for $i\neq j$ are not zero-divisors
in $R$. Let $k,n_1,\ldots,n_r\in\nn$, and
$f(x)=\sum_{l\geq k}\alpha_lx^l\in R[[x]]$ be such
power series that $f(t_i+y)=\sum_{j\geq n_i}\beta_{i,j}y^j$, for some $\beta_{i,j}\in R$.
Then $\alpha_k$ is divisible by $\prod_{i=1}^r t_i^{n_i}$.
\end{lem}

\begin{proof}
Let $(n_1,\ldots,n_r)$ be an $r$-tuple of natural numbers, and $m\geq
N:=\sum_{j=1}^rn_j$. Denote $N_i:=\sum_{j=1}^in_j$.
For $1\leq u\leq N$, denote as $|u|$ such $1\leq i\leq r$ that
$N_{i-1}<u\leq N_i$, and as
$\ov{u}$ the difference $u-N_{i-1}-1$ (so, $0\leq\ov{u}<n_{|u|}$).
Consider $N\times m$ matrix
$A(n_1,\ldots,n_r;m)=\{a_{u,v}\}$, where
$a_{u,v}=\binom{v-1}{\ov{u}}t_{|u|}^{v-1-\ov{u}}$.

\begin{claim} \label{minors}
\begin{itemize}
\item[$(1)$\ ]
All $N\times N$ minors of the matrix $A(n_1,\ldots,n_r;m)$ are divisible
by\\
$\prod_{i>j}(t_i-t_j)^{n_i\cdot n_j}$.
\item[$(2)$\ ]
If $m=N$, then
$\op{det}(A(n_1,\ldots,n_r;m))=\prod_{i>j}(t_i-t_j)^{n_i\cdot n_j}$.
\end{itemize}
\end{claim}

\begin{proof}
Let us prove both statements by induction on $N$.
For $N=0$ there is nothing to prove.
Notice, that elementary transformations on rows and
columns do not change the ideal generated by minors, while multiplication of
some row by $\lambda$ multiplies this ideal by $\lambda$.
Perform the following transformations (in the prescribed order):

1) For all $m>v\geq 1$ subtract $v$-th column times $t_1$ from $(v+1)$-st.

2) For all $2\leq i\leq r$ subtract $1$-st row from $(N_{i-1}+1)$-st, and
   divide the result (the new $(N_{i-1}+1)$-st row) by $(t_i-t_1)$.

3) For all $2\leq i\leq r$, for all $N_{i-1}+1\leq u<N_i$, subtract
   $u$-th row from $(u+1)$-st, and divide the result (the new $(u+1)$-st row)
   by $(t_i-t_1)$.

The result will be the block matrix $1\times A(n_1-1,n_2,\ldots,n_r;m-1)$.
Thus, the ideal generated by the minors of $A(n_1,\ldots,n_r;m)$ is
$\prod_{i=2}^r(t_i-t_1)^{n_i}$ times the ideal generated by the minors
of $A(n_1-1,n_2,\ldots,n_r;m-1)$. Induction step is proven.
\Qed
\end{proof}

Since for $f(x)=\sum_{l\geq k}\alpha_l x^l$, we have
$f(t_i+y)=\sum_{j\geq 0}\beta_{i,j}y^j$ with $\beta_{i,j}=0$, for
$j<n_i$, we get $n_i$ equations:
$$
\sum_{l\geq k}\binom{l}{w}t_i^{l-w}\alpha_l=0,\,\,\,0\leq w<n_i.
$$
Performing elementary transformations with rows and dividing by $t_i$
(which is possible since $t_i$ is not a zero divisor), we get the
equivalent system:
$$
\sum_{l\geq k}\binom{l-k}{w}t_i^{l-k-w}\alpha_l=0,\,\,\,0\leq w<n_i.
$$
Combining all such systems for $1\leq i\leq r$, we get the system with
the matrix $A(n_1,\ldots,n_r;\infty)$.
Let $B$ be such a matrix that
$$
B\cdot A(n_1,\ldots,n_r;N)=\op{det}(A(n_1,\ldots,n_r;N))\cdot Id=
\prod_{i>j}(t_i-t_j)^{n_i\cdot n_j}\cdot Id.
$$
Since $(t_i-t_j)$ are not zero divisors,
the system with the matrix $A(n_1,\ldots,n_r;\infty)$ is
equivalent to the system with the matrix $B\cdot A(n_1,\ldots,n_r;\infty)$.
This shows that $\prod_{i>j}(t_i-t_j)^{n_i\cdot n_j}\cdot\alpha_k$ can be
expressed as a linear combination of $\alpha_l$, with $l\geq k+N$.
And, by the Kramer's rule, the coefficient at $\alpha_{k+m}$ will be (minus) the minor
$M_{2,3,\ldots,N,m+1}$ of the matrix $A(n_1,\ldots,n_r;\infty)$.
Dividing by $t_i$ and performing elementary transformations with
rows we see that this minor is equal to $\prod_{i=1}^rt_i^{n_i}$ times
the minor $M_{1,2,\ldots,N-1,m}$. It follows from the Claim \ref{minors}
that $M_{1,2,\ldots,N-1,m}$ is divisible by
$\prod_{i>j}(t_i-t_j)^{n_i\cdot n_j}$. Using again the fact that $(t_i-t_j)$
is not a zero divisor, we express $\alpha_k$ as a linear combination of
$\alpha_l$, $l\geq k+N$, where all the coefficients are divisible by
$\prod_{i=1}^rt_i^{n_i}$.
\Qed
\end{proof}

\begin{rem}
Above we are working with infinite linear relations, but it follows from our condition
on $t_i$'s that the respective sums do converge.
\end{rem}

\begin{defi}
\label{conthom}
Let $B=\lim\limits_{\low}^{}B_n$ be a commutative ring.
We say that $\sigma:B[[x]]\row B[[x]]$ is a {\it continuous
$B$-homomorphism}, if, for any $\eps(x)\in B[[x]]$, one has $\sigma(\eps(x))=\eps(x^{\sigma})$, where
$x^{\sigma}=\sum_{j\geq 0}\lambda^{\sigma}_jx^j$, and
$B$ is complete with respect
to $\lambda^{\sigma}_0$ (that is,
$(\lambda^{\sigma}_0)^n=0\in B_n$).
If $\lambda^{\sigma}_1\in B$ is invertible, then such map has
an inverse (also {\it continuous $B$-homomorphism}), and
we call it {\it continuous $B$-automorphism}.
Continuous $B$-automorphisms form a group $\op{Aut}^c_B(B[[x]])$ under composition.
\end{defi}

\begin{defi}
\label{contgract}
Let $G$ be a group. We say that $G$ {\it acts continuously}
on $B[[x]]$ if we are given a group homomorphism
$\rho:G\row\op{Aut}^c_B(B[[x]])$.
\end{defi}

\begin{thm}\label{G}
Let $\rho:G\row\op{Aut}^c_B(B[[x]])$ be a continuous action of a
finite group $G$ on $B[[x]]$. Suppose that for all $g\in G\backslash e$, the
elements $\lambda^g_0\in B$ are not zero divisors. Then
the subring of invariants is given by:
$$
B[[x]]^{G}=B[[\prod_{g\in G}x^g]].
$$
\end{thm}

\begin{proof}
Let $\ffi(x)\in B[[x]]^G$ be invariant power series.
Let $\ffi(x)=\sum_{j\geq n}\alpha_jx^j$. Consider the power series
$x^g\in B[[x]],\,g\in G$. Let us denote $t_g:=\lambda^g_0$.

Let $g,h\in G$ be different elements.
Then $x^g=(x^h)^{h^{-1}g}$.
Thus, $t_g=t_h+\sum_{j\geq 1}\lambda^{h}_jt_{h^{-1}g}^j$, and, up to
an invertible factor, $(t_g-t_h)$ is equal to $t_{h^{-1}g}$, which is
not a zero divisor.

Since $\ffi(x)=\ffi(x^g)=\ffi(t_g+y_g)$, where
$y_g=\sum_{j\geq 1}\lambda^g_jx^j$, and the ideal $(x)$ generated by $x$
coincides with the ideal $(y_g)$ generated by $y_g$ (since $\lambda^g_1$ is
invertible), we have that $\ffi(t_g+y_g)=\sum_{j\geq n}\beta_{g,j}y_g^j$.
It follows from Lemma \ref{del} that $\alpha_n$ is divisible by
$\prod_{g\in G\backslash e} t_g^n$. Let
$\gamma_n\cdot \prod_{g\in G\backslash e} t_g^n =\alpha_n$.
Then $\psi(x):=\ffi(x)-\gamma_n\cdot(\prod_{g\in G}x^g)^n$ is also
invariant, and belongs to $(x)^{n+1}$.

Thus, any power series invariant under $G$ can be expressed as a power
series in $(\prod_{g\in G}x^g)$. Theorem is proven.
\Qed
\end{proof}

\begin{cor} \label{z/p}
Let $B$ be commutative ring, with a continuous action of
$\zz/p\cdot\la\sigma\ra$ ($p$-prime)
on $B[[x]]$. Suppose that $t_{\sigma}\in B$ is not a zero divisor.
Then
$$
B[[x]]^{\zz/p}=B[[\prod_{i=0}^{p-1}x^{\sigma^i}]].
$$
\end{cor}

\begin{proof}
It is sufficient to observe that $t_{g^m}$ is divisible by $t_g$.
Since, for arbitrary non-zero element $g\in\zz/p\cdot\la\sigma\ra$, there is $m$
such that $g^m=\sigma$, we have that $t_{g}$ is not a zero divisor,
and we can apply Theorem \ref{G}.
\Qed
\end{proof}

\section{T.tom Dieck-style Steenrod operations}
\label{TtD}

Let $R$ be commutative ring with the formal group law on it (or, which is
the same, with the ring homomorphism $\eps:\laz\row R$), and $p$ be
prime number.
Define $B:=R[[t]]/(\frac{p\cdot_F t}{t})$, where
$p\cdot_F t\in R[[t]]$ is $p$ times $t$ in the sense of the formal
group law. Then $B$ is complete with respect to $t$.

Let us define the continuous action of $\zz/p\cdot\la\sigma\ra$ on $B[[x]]$
by the formula: $x^{\sigma}=(x+_F t)\in B[[x]]$.
Notice, that $t_{\sigma}=t$, and
$\lambda^{\sigma}_1=1+\sum_{i\geq 1}a_{i,1}t^i$ is invertible.
Suppose $p$ is not a zero divisor in $R$. Then $t$ is not a zero divisor
in $B$, and we can apply Corollary \ref{z/p}.
We get:

\begin{cor} \label{omegaz/p}
In the above situation,
$$
B[[x]]^{\zz/p}=B[[\prod_{i=0}^{p-1}(x+_F i\cdot_F t)]].
$$
\end{cor}

\begin{prop} \label{x+y}
Let $R$ be commutative ring with the formal group law, and
$B:=R[[t]]/(\frac{p\cdot_F t}{t})$, then there exists
power series in two variables $G(u,v)$ with coefficients
in $B$ such that
$$
\prod_{i=0}^{p-1}(x+_Fy+_F i\cdot_F t)=G(\prod_{i=0}^{p-1}(x+_F i\cdot_F t),
\prod_{i=0}^{p-1}(y+_F i\cdot_F t)).
$$
\end{prop}

\begin{proof}
Clearly, it is sufficient to prove this statement for $R=\laz$ with
the universal formal group law on it. In this case, $p$ is not a zero divisor
in $R$, and so $t$ is not a zero-divisor in $B$.
Consider the action of $\zz/p\times\zz/p$ on $B[[x,y]]$ given by
$x^{\sigma}=(x+_F t)$, $y^{\sigma}=y$, $x^{\tau}=x$, $y^{\tau}=(y+_F t)$.
Clearly, $\prod_{i=0}^{p-1}(x+_Fy+_F i\cdot_F t)\in B[[x,y]]^{\zz/p\times\zz/p}$.
Applying Corollary \ref{omegaz/p}, we obtain:
$B[[x,y]]^{\zz/p\times 1}=B[[\prod_{i=0}^{p-1}(x+_F i\cdot_F t),y]]=:C[[y]]$,
and $t$ is not a zero divisor in $C$ either. Thus,
$$
B[[x,y]]^{\zz/p\times\zz/p}=B[[\prod_{i=0}^{p-1}(x+_F i\cdot_F t),
\prod_{i=0}^{p-1}(y+_F i\cdot_F t)]].
$$
Proposition is proven.
\Qed
\end{proof}

In the above situation (with $t\in B$ not a zero-divisor), consider the power series
$\alpha(x)=x\prod_{i=1}^{p-1}(x+_F i\cdot_F t)\in B[[x]]$.
The first term of this power series is $(\prod_{i=1}^{p-1} i\cdot_F t)\cdot x$.
Notice, that $i\cdot_F t$, for
$i=1,\ldots,p-1$, are invertible in $B[t^{-1}]$. Thus, there exists the inverse
power series $\beta(y)\in B[t^{-1}][[y]]$ such that $\beta(\alpha(x))=x$.

Consider the twisted formal group law
$F^{\alpha}$ given by
$$
F^{\alpha}(u,v):=\alpha(F(\beta(u),\beta(v)))\in
B[t^{-1}][[u,v]]
$$

\begin{prop}
The formal group law $F^{\alpha}$ has coefficients in $B$.
\end{prop}

\begin{proof}
This follows immediately from Proposition \ref{x+y}.
\Qed
\end{proof}

Combining this with Theorem \ref{neobrB0} we obtain
Steenrod operations of T.tom Dieck style for $\Omega^*$:

\begin{thm}
\label{TtDieck}
For each prime $p$ there exists (unique) multiplicative operation
$$
Sq:\Omega^*\row\Omega^*[[t]]/\left(\frac{p\cdot_{\Omega}t}{t}\right)
$$
with $\gamma_{Sq}(x)=x\cdot\prod_{0<i<p}(x+_{\Omega}i\cdot_{\Omega}t)$.
\end{thm}

\section{Symmetric operations for all primes}
\label{Sp}

\subsection{Construction}

By comparing the respective morphisms of formal group laws and using Theorem \ref{multFGL} we obtain a
commutative diagram relating Steenrod operations of D.Quillen and T.tom Dieck types:
$$
\xymatrix @-0.2pc{
\Omega^* \ar @{->}[r]^(0.4){St(\ov{i})} \ar @{->}[d]_(0.5){Sq}&
\Omega^*[\iis^{-1}][[t]][t^{-1}] \ar @{->}[d]^(0.5){}\\
\Omega^*[[t]]/(\frac{p\cdot_{\Omega}t}{t}) \ar @{->}[r]_(0.5){} &
\Omega^*[[t]][t^{-1}]/(p\cdot_{\Omega}t),
}
$$
where $\ov{i}$ is any choice of coset representatives.

Since the target of $Sq$ has no negative powers of $t$,
and the $t^0$-component of it coincides with the $p$-th power $\square^p$,
the commutativity of the above diagram shows that the {\it non-positive part}
of $(\square^p-St(\ov{i}))$ is divisible by $[p]_{\Omega}:=\frac{p\cdot_{\Omega}t}{t}$.
I should point out that this fact itself can be proven without Steenrod operations
of T.tom Dieck type, and without the Theorem \ref{multFGL} (or methods of \cite{SU}).
But what is much deeper,
it appears that one can divide "canonically", and the quotient is what we call
{\it Symmetric operation}.

\begin{thm}
\label{SOp}
There is unique operation $\Phi(\ov{i}):\Omega^*\row\Omega^*[\iis^{-1}][t^{-1}]$ such that
$$
(\square^p-St(\ov{i})-[p]_{\Omega}\cdot\Phi(\ov{i})):
\Omega^*\row\Omega^*[\iis^{-1}][[t]]t.
$$
\end{thm}

\begin{proof}
Consider $A^*=\Omega^*$, $B^*=\Omega^*[\iis^{-1}][[t]][t^{-1}]$,
$C^*=\Omega^*[\iis^{-1}][t^{-1}]$, and
$D^*=\Omega^*[\iis^{-1}][[t]]t$. Take
$b=[p]_{\Omega}$. Then
$m_b:\laz[\iis^{-1}][t^{-1}]\row \laz[\iis^{-1}][t^{-1}]$ is
injective as $\laz$ is an integral domain.
Consider $G=(\square^p-St(\ov{i})):A^*\row B^*$. Then
$\pi_C\circ G=St(\ov{i})_{\leq 0}:A^*\row B^*$ - the {\it non-positive part} of
$(\square^p-St(\ov{i}))$
corresponding to monomials in $t$ of non-positive degree.
By the above diagram, $image(\pi_C\circ G)\subset image(m_b)$, and by Corollary \ref{div},
there is unique operation $\Phi(\ov{i}):A^*\row C^*$ such that
$(\square^p-St(\ov{i})-b\cdot\Phi(\ov{i})):A^*\row D^*$.
\Qed
\end{proof}

Some traces of the $MU$-analogue of this operation were used by D.Quillen in \cite{Qu71},
and they provide the main tool of the mentioned article.

In Algebraic Cobordism the described operation appeared originally in the works \cite{so1}
and \cite{so2} of the author in the case $p=2$ in a different form. Namely, in the form of "slices",
which were constructed geometrically.
Only substantially later the author had realized that these slices can be combined
into the "formal half" of the "negative part" of some multiplicative operation, which
had a power series $\gamma=x\cdot(x-_{\Omega}t)$ reminiscent of a Steenrod operation
in Chow groups mod $2$.

Out of our operation $\Phi(\ov{i})$ we would like to produce some maps from $\Omega^*$ to itself.
The natural approach would be to consider the
coefficients of it at particular monomials $t^{-n}$, or, which is close,
$\operatornamewithlimits{Res}_{t=0}\frac{t^n\cdot\Phi(\ov{i})\omega_t}{t}$.
Here $\omega_t$ is the canonical invariant (w.r.to our FGL)
form $([\pp^0]+[\pp^1]t+[\pp^2]t^2+\ldots)dt$ - see
\cite[Sect. 7.1]{SU},
and $\operatornamewithlimits{Res}_{t=0}$ is the coefficient at $\frac{dt}{t}$.
And, if one thinks about it, there is no point restricting oneself to monomials, so
one can consider
$$
\Phi(\ov{i})^{q(t)}:=
\operatornamewithlimits{Res}_{t=0}\frac{q(t)\cdot\Phi(\ov{i})\omega_t}{t},
$$
where $q(t)=q_0+q_1t+q_2t^2+\ldots\in\laz[[t]]$ is any power series.
Of course, there are various relations among these slices which bind them together
into the operation $\Phi(\ov{i})$.
For $p=2$, these are exactly the Symmetric operations $\Phi^{q(t)}$ of \cite{so2}:

\begin{prop}
\label{SOoldnew}
In the case $p=2$, with $\ov{i}=\{-1\}$,
for any power series as above, we have:
$$
\Phi(\ov{i})^{q(t)}=\Phi^{q(t)}.
$$
\end{prop}

\begin{proof}
By Theorem \ref{MAIN}, it is sufficient to compare our operations on cellular spaces
$(\pp^{\infty})^{\times r}$. It follows from \cite[Propositions 2.13 and 2.15]{so2} that,
for any $q(t)\in\laz[[t]]$,
$\Phi^{([2]_{\Omega})q(t)}=q(0)\cdot\square-
\operatornamewithlimits{Res}_{t=0}\frac{q(t)\cdot\Psi\cdot\omega_t}{t}$,
where $\Psi:\Omega^*\row\Omega^*[[t]][t^{-1}]$ is the multiplicative operation with
$\gamma_{\Psi}=x\cdot(x-_{\Omega}t)$. Thus, $\Psi=St(\ov{i})$, where $\ov{i}=\{-1\}$.
And so, $\Phi^{([2]_{\Omega})q(t)}=\Phi(\ov{i})^{([2]_{\Omega})q(t)}$, by the
very definition of the latter.
But on cellular spaces, the multiplication by $[2]_{\Omega}$ is injective,
as $\Omega^*$ of such a space is a free $\laz$-module (for $(\pp^{\infty})^{\times r}$,
it is a direct consequence of the
{\it projective bundle axiom} - see Subsect. \ref{FGL}). Hence, $\Phi(\ov{i})^{q(t)}=\Phi^{q(t)}$
as well (cf. \cite[proof of Corollary 2.17]{so2}).
\Qed
\end{proof}

The cases $p=2$ and $3$ are special, since we can choose our representatives $\ov{i}$ to
be invertible in $\zz$. For $p=2$, we have two such choices: $\{1\}$, or $\{-1\}$ (in \cite{so2},
$\{-1\}$ was "chosen"). For $p=3$, the choice is canonical: $\{1,-1\}$.
Thus, we get integral operations $\Phi(\ov{i}):\Omega^*\row\Omega^*[t^{-1}]$.
And, for arbitrary $p$, representatives can be chosen as the powers of some fixed prime $l$
(generating $(\zz/p)^*$), so that only one prime would be inverted. Moreover, this prime
can be selected in infinitely many ways, so, in a sense, the picture is as good as integral.

\subsection{Some properties}

First of all, we should mention the Riemann-Roch type result which describes how our operations behave
with respect to regular embeddings.

Let $\cn$ be a vector bundle on $X$ with $\Omega$-roots $\lambda_1,\ldots,\lambda_d$.
Denote as $\che{\cn}{\ov{i}}$ the element
$$
\prod_{l=1}^d\left(\frac{\gamma_{St(\ov{i})}(x)}{x}\right)(\lambda_l)=
\prod_{l=1}^d\prod_{j=1}^{p-1}(\lambda_l+_{\Omega}i_j\cdot_{\Omega}t)=
\prod_{j=1}^{p-1}c^{\Omega}(\cn)(i_j\cdot_{\Omega}t)
\in\Omega^*(X)[[t]],
$$
where $c^{\Omega}(\cn)(t)=\prod_{i=1}^d(t+_{\Omega}\lambda_l)$.
Analogously, one can define the Chow group versions: $c(\cn)(t)$ - the usual total Chern power series,
and $\cheCH{\cn}{\ov{i}}$ (where the formal addition is substituted by the usual one).
Then we have (cf. \cite[Proposition 3.1]{so2}):

\begin{prop}
\label{RR}
Let $X\stackrel{f}{\row}Y$ be a regular embedding of smooth quasi-projective varieties with the normal
bundle $\cn_f$, and $q(t)=q_0+q_1t+q_2t^2+\ldots\in\Omega^*(X)[[t]]$. Then
$$
\Phi(\ov{i})^{q(t)}(f_*(u))=f_*(\Phi(\ov{i})^{q(t)\cdot\che{\cn_f}{\ov{i}}}(u)).
$$
\end{prop}

\begin{proof}
Consider two operations:
$$
\Omega^*(Z)\stackrel{G,\wt{G}}{\lrow}\Omega^*(Z\times(\pp^{\infty})^{\times d}),
$$
where $G(v):=\Phi(\ov{i})^{q(t)}(\pi^*(v)\cdot\prod_{l=1}^d z_l)$, and
$\wt{G}(v):=\Phi(\ov{i})^{\wt{q}(t)}(\pi^*(v))$, where
$z_l=c_1^{\Omega}(\co(1)_l)$,
$\pi:X\times(\pp^{\infty})^{\times r}\row X$
is the projection, and
$\wt{q}(t)=q(t)\cdot\prod_{l=1}^d\gamma_{St(\ov{i})}(z_l)$.
We can write $\wt{q}(t)$ as $\wt{\wt{q}}(t)\cdot\prod_{l=1}^dz_l$, where
$\wt{\wt{q}}(t)=q(t)\cdot \prod_{l=1}^d\left(\frac{\gamma_{St(\ov{i})}(x)}{x}\right)(z_l)$.
Operations $G$ and $\wt{G}$ coincide on all $(\pp^{\infty})^{\times m}$.
This follows from: the fact that for the multiplicative
operation $H=St(\ov{i})$ we have: $H(\pi^*(v)\cdot\prod_{l=1}^dz_l)=
H(\pi^*(v))\cdot \prod_{l=1}^d\gamma_{St(\ov{i})}(z_l)$; and the fact that on cellular spaces such as $(\pp^{\infty})^{\times m}$, $\Phi(\ov{i})^{r(t)}$ can be written as
$\operatornamewithlimits{Res}_{t=0}
\frac{(r(0)\square^p-r(t)\cdot H)\omega_t}{p\cdot_{\Omega}t}$, where the division
by $p\cdot_{\Omega} t$ is uniquely defined (as $\Omega^*$ of such spaces is a free $\laz$-module).
By Theorem \ref{MAIN}, $G=\wt{G}$.
It remains to apply the general (non-additive) Riemann-Roch Theorem - see \cite[Proposition 5.19]{PO}.
Recall, that due to the Projective Bundle axiom, any element of
$\Omega^*(X\times (\pp^{\infty})^d)$ can be written as a (unique) $\Omega^*(X)$-power series $\alpha(z^A_1,\ldots,z^A_d)$
in the $1$-st Chern classes of the bundles $\co(1)$ from components. The superscripts $A$ is introduced
to indicate that we are dealing with the source of the operation. When we apply any operation $F$ to
$\alpha(z^A_1,\ldots,z^A_d)$, we obtain again some $\Omega^*(X)$-power series in
$z^B_1,\ldots,z^B_d$ (the same $1$-st Chern
classes, but in the target) which we denote $F(\alpha(z^A_1,\ldots,z^A_d))(z^B_1,\ldots,z^B_d)$.
Now we can plug whatever we want instead of the formal $B$-variables.
This way, we can describe what happens to $F$ under regular push-forwards. Namely, by
\cite[Proposition 5.19]{PO}, the condition $(b_{ii})$ of \cite[Definition 5.5]{PO} is satisfied, and so
we have, for any $u\in\Omega^*(X)$,
$$
F(f_*(u))=f_*\operatornamewithlimits{Res}_{s=0}
\frac{F(\prod_{l=1}^dz^A_l\cdot u)(z_l^B=s+_{\Omega}\lambda_l|_{l\in\ov{d}})\cdot\omega_s}
{\prod_{l=1}^d(s+_{\Omega}\lambda_l)\cdot s},\hspace{5mm}\text{which implies}:
$$
\begin{equation*}
\begin{split}
&\Phi(\ov{i})^{q(t)}(f_*(u))=f_*\operatornamewithlimits{Res}_{s=0}
\frac{\Phi(\ov{i})^{\wt{q}(t)}(u)(z_l=s+_{\Omega}\lambda_l|_{l\in\ov{d}})\cdot\omega_s}
{\prod_{l=1}^d(s+_{\Omega}\lambda_l)\cdot s}=\\
&f_*\Phi(\ov{i})^{\wt{\wt{q}}(t)}(u)(z_l=\lambda_l|_{l\in\ov{d}})=
f_*\Phi(\ov{i})^{q(t)\cdot\che{\cn_f}{\ov{i}}}(u).
\end{split}
\end{equation*}
(Notice, that $\Phi(\ov{i})^{\wt{q}(t)}$ is $\Omega^*(X\times(\pp^{\infty})^{\times d})$-linear in
$\wt{q}$, by definition, and so, $\Phi(\ov{i})^{\wt{q}(t)}(u)(z_l|_{l\in\ov{d}})=
\left(\prod_{l\in\ov{d}}z_l\right)\cdot\Phi(\ov{i})^{\wt{\wt{q}}(t)}(u)(z_l|_{l\in\ov{d}})$.
)
\Qed
\end{proof}

As an application of the above result we obtain that Symmetric operations provide
obstructions for a cobordism class
to be presented by a class of an embedding (cf. \cite[Proposition 3.2]{so2}).

\begin{prop}
\label{emb}
Let $V\stackrel{v}{\row}X$ be a regular embedding. Then $\Phi(\ov{i})(v)=0$.
\end{prop}

\begin{proof}
By dimensional considerations, $\Phi(\ov{i})(1_{\op{Spec}(k)})=0$. Since $\Phi(\ov{i})$ is
an operation, we have: $\Phi(\ov{i})(1_{V})=0$, for all $V$. Then it follows from
Proposition \ref{RR} that $\Phi(\ov{i})(v)=0$.
\Qed
\end{proof}

In another direction, our Symmetric operations help to study $\laz$-torsion and relations in $\Omega^*(X)$.
For this purpose, let us introduce a close relative of the invariant of M.Rost.

\begin{defi}
\label{nuPi}
Having $p$ and $\ov{i}$, and a smooth projective $U$ of positive dimension, define
$$
\eta_{p,\ov{i}}(U):=-\frac{\ddeg(\cheCH{-T_U}{\ov{i}}_0)}{p}\in\zz[\iis^{-1}],
$$
where
$T_U$ is the tangent bundle, and
$\cheCH{-T_U}{\ov{i}}_0$ is the zero-dimensional ($t^{-p\ddim(U)}$-) component of the
$\prod_{j=1}^{p-1}c(-T_U)(i_j\cdot t)$ (here $c(\cn)(t)$ is the total Chern power series -
see above).
\end{defi}

The fact that $\eta_{p,\ov{i}}$ is well-defined follows from the following result:

\begin{prop}
\label{F1}
Let $U$ be smooth projective variety of dimension $n>0$, and $[U]\in\laz$ be the respective class. Then
$$
\ddeg(\Phi(\ov{i})^{t^{pn}}([U]))=\eta_{p,\ov{i}}(U)\in\zz[\iis^{-1}].
$$
\end{prop}

\begin{proof}
We know that $\Phi(\ov{i})^{t^{pn}}([U])\in\laz_0[\iis^{-1}]=\zz[\iis^{-1}]$, and
\begin{equation*}
\begin{split}
&p\cdot\Phi(\ov{i})^{t^{pn}}([U])=
-\operatornamewithlimits{Res}_{t=0}\frac{t^{pn}\cdot
St(\ov{i})([U])\omega_t}{t}=
-\operatornamewithlimits{Res}_{t=0}\frac{t^{pn}\cdot
St(\ov{i})(\pi_*(1_U))\omega_t}{t}=\\
&-\operatornamewithlimits{Res}_{t=0}\frac{t^{n}\cdot
-\pi_*(\cheCH{-T_U}{\ov{i}})\omega_t}{t}=
-\op{deg}(\cheCH{-T_U}{\ov{i}}_0),
\end{split}
\end{equation*}
where $\pi:U\row\op{Spec}(k)$ is the projection.
\Qed
\end{proof}

\begin{rem}
\label{rost}
The invariant of M.Rost is defined a bit differently. It is the degree of the zero cycle
$c_1(\cl)^{np}$, where $\cl$ is the standard linear bundle on a smooth, but not proper variety $C^pU\backslash\Delta(U)$ (here $\cl$ is a quotient of $\co$ on $U^{\times p}\backslash\Delta(U)$
by the $\zz/p$-action). Such a degree is well-defined in $\zz/n_U$, where $n_U$ is the greatest common
divisor of the degrees of closed points on $U$.
\end{rem}

Composing the Total Landweber-Novikov operation with the canonical morphism of theories
$pr:\Omega^*\row\op{CH}^*$ and evaluating on a point, we obtain the Hurewitz map
$$
\laz\row\zz[b_1,b_2,\ldots].
$$
Coefficients at particular monomials give us characteristic numbers
$\chi_{\ov{b}^{\ov{J}}}:\laz\row\zz$. On the class of a smooth projective variety $U$ these
can be alternatively computed as degrees of zero-cycles given by certain polynomials in
Chern classes of $-T_U$.
More precisely, the respective zero-cycle will be the coefficient at $\ov{b}^{\ov{J}}$ in the
product $\prod_{\lambda\in\Lambda}(1+b_1\lambda+b_2\lambda^2+\ldots)$, where
$\Lambda$ is the set of Chow-roots of $-T_U$.
Let $I(p)\subset\laz$ be the ideal consisting of classes whose all characteristic
numbers are divisible by $p$. Due to results of Landweber, it is stable under the action of
the Landweber-Novikov operations (which is obvious as soon as we know that Landweber-Novikov operations form an algebra, and that characteristic numbers are
exactly the results of various Landweber-Novikov operations applied to the class).
I recall, that an element $x\in\laz_{p^r-1}$ of dimension $(p^r-1)$ is called a $\nu_r$-element, if
it belongs to $I(p)$, and the (only) additive characteristic number of it is not divisible by $p^2$.
After projecting to $BP$-theory such an element can be chosen as a polynomial generator of the
coefficient ring - see \cite{Wi}.

Since $\zz[\iis^{-1}]/p=\zz/p$, we can compare $\ov{\eta}_{p,\ov{i}}(U)\in\zz/p$
for different $\ov{i}$.

\begin{prop} {\rm (cf. \cite{R})}
\label{il}
\begin{itemize}
\item[$(1)$ ]
Let $[U]\in I(p)$. Then, for all $\ov{i}$ and $\ov{l}$,
$$
\ov{\eta}_{p,\ov{i}}(U)=\ov{\eta}_{p,\ov{l}}(U);
$$
\item[$(2)$ ] Let $U$ has no zero cycles of degree prime to $p$.
Then, up to sign, $\ov{\eta}_{p,\ov{i}}(U)$ coincides with
$\ov{\eta}_p(U)$ - the invariant of M.Rost mod $p$;
\item[$(3)$ ] If $[U]$ is a $\nu_r$-element in $\laz$, then $\ov{\eta}_p(U)\neq 0$.
\end{itemize}
\end{prop}

\begin{proof}
(1)
We know that $(\gamma_{St(\ov{i})}-\gamma_{St(\ov{l})})$ is divisible by $[p]_{\Omega}$.
But any multiplicative operation $G:\Omega^*\row B^*$ with $\gamma_G(x)=b_0x+b_1x^2+\ldots$,
where $b_0$ is invertible, is a generalized specialization of the {\it Total Landweber-Novikov
operation}. In particular, $G$ is a linear combination of specializations of the individual
Landweber-Novikov operations with coefficients - monomials in $b_0^{\pm 1}$ and $b_k/b_0,\,k>0$.
Thus $(St(\ov{i})-St(\ov{l}))$ will be a linear combination of the Landweber-Novikov operations
with coefficients divisible by $[p]_{\Omega}$ (note, that this linear combination will depend
on the component $\Omega^n$ on which it acts, as our operation is unstable).
This implies that the difference of two Symmetric operations
$$
(\Phi(\ov{i})-\Phi(\ov{l})):\Omega^*\row\Omega^*[\iis^{-1},\lls^{-1}][t^{-1}]
$$
is a linear combination of the Landweber-Novikov operations.
It remains to observe that the ideal $I(p)$ is stable under the latter, and
the zero-dimensional component of $I(p)$ is $p\cdot\zz$.

(2) If $U$ has no zero cycles of degree prime to $p$, then we have the surjection
$\zz/n_U\twoheadrightarrow\zz/p$ (notice, that, in particular, $[U]\in I(p)$).
It follows from computations of M.Rost in
\cite{R} that $-\iis^{\ddim(U)}\cdot\ov{\eta}_{p,\ov{i}}(U)=\ov{\eta}_p(U)$.
Remains to observe that $\iis\equiv -1\,(\,mod\,p)$.

(3) Since $\prod_{j=1}^{p-1}(x+jt)\equiv\prod_{j=1}^{p-1}(x+\eps^jt)\equiv(x^{p-1}+t^{p-1})\,\,(\,mod\,p)$,
where $\eps$ is a primitive root of $1$ of degree $(p-1)$,
by the considerations from part $(1)$, we obtain that the characteristic number for
$\cheCH{-T_U}{\ov{i}}$ can be substituted by the one for
$b(-T_U)$, where $b(\cn)(t)=\prod_{j=1}^{p-1}c(\cn)(\eps^j\cdot t)$.
Notice, that $(b(-T_U))_{0}$ is the characteristic number corresponding to the partition
$(p-1,p-1,\ldots,p-1)$, or in other words, to the Landweber-Novikov operation
$S_{L-N}^{b_{p-1}^d}$, where $d=\frac{\ddim(U)}{p-1}$.
The latter operation is a component of the multiplicative operation
$$
\sum_{m\geq 0}b_{p-1}^m\cdot S_{L-N}^{b_{p-1}^m}:\Omega^*\row\Omega^*[b_{p-1}]
$$
(a specialization of the Total Landweber-Novikov operation:
$b_i\mapsto 0$, $i\neq p-1$). And operations $S_{L-N}^{b_{p-1}^m}$ descend to
Steenrod operations on $\op{CH}^*/p$, which implies that $\chi_{b_{p-1}^m}$
is always divisible by $p$, for $m>0$.
It follows that our characteristic number $\chi_{b_{p-1}^d}$
is divisible by $p^2$ on every decomposable element of the Lazard ring.
Hence, modulo $p^2$ it is the same
(up to a factor invertible modulo $p$)
on each $\nu_r$-element (since such an element can be chosen as a polynomial generator of
$\laz\otimes\zz_{(p)}$).
It is sufficient to compute it on the class of a hypersurface $Q$ of
degree $p$ in $\pp^{p^r}$. And for such a hypersurface,
$\ov{\chi}_{b_{p-1}^d}(Q)=(-1)^r\dbinom{\frac{p^{r+1}-1}{p-1}}{\frac{p^{r}-1}{p-1}}\neq 0\in\zz/p\zz$.
\phantom{a}\hspace{5mm}\Qed
\end{proof}

The above invariants can be applied to the computation of Chow traces $\phi(\ov{i})^{q(t)}$
of Symmetric operations, that is, compositions
$\Omega^*\stackrel{\Phi(\ov{i})^{q(t)}}{\lrow}\Omega^*\stackrel{pr}{\lrow}\op{CH}^*$.
In analogy with slices of $\Phi$,
for any $f(t)\in\laz[[t]][t^{-1}]$, let us denote as $St(\ov{i})^{f(t)}$ the operation
$\operatornamewithlimits{Res}_{t=0}\frac{f(t)\cdot St(\ov{i})\cdot\omega_t}{t}$, and as
$st(\ov{i})^{f(t)}$ the Chow trace of it.

The following result shows that if we are given $u\cdot v$, where $u\in\laz_{>0}$, then using
Symmetric operations we can obtain, if not $v$ itself, at least, some multiple of it's Chow trace
$pr(v)$. And the coefficient involved is invertible modulo $p$ in interesting situations
(which will not be the case if one uses Landweber-Novikov, or Steenrod operations instead).

\begin{prop}
\label{uv}
Let $v\in\Omega^*(X)$ and $u=[U]\in\laz_{>0}$.
Let $q(t)\in\op{CH}^*(X)[[t]]$. Then
$$
\phi(\ov{i})^{q(t)}(u\cdot v)=\eta_{p,\ov{i}}(U)\cdot st(\ov{i})^{q(t)t^{-p\cdot\ddim(U)}}(v).
$$
In particular, if $k=p\ddim(u)-(p-1)codim(v)$ is positive, then
$$
\phi(\ov{i})^{t^k}(u\cdot v)=\eta_{p,\ov{i}}(U)\cdot\iis^{codim(v)}\cdot pr(v).
$$
\end{prop}

\begin{proof}
Since both sides of the equation are $\op{CH}^*(X)$-linear on $q(t)$, we can assume that $q(t)\in\zz[[t]]$.
Then it follows from Theorem \ref{MAIN} that it is sufficient to compare our operations
(as operations on $v$) on cellular
spaces $(\pp^{\infty})^{\times l}$. On such a space, using the fact that $St(\ov{i})$ is multiplicative,
we can write LHS as
\begin{equation*}
\begin{split}
& -pr \operatornamewithlimits{Res}_{t=0}\frac{St(\ov{i})(u\cdot v)\cdot q(t)\cdot\omega_t}
{p\cdot_{\Omega} t}=
-\operatornamewithlimits{Res}_{t=0}\frac{pr(St(\ov{i})(u)\cdot St(\ov{i})(v)\cdot q(t)\cdot\omega_t)}
{p\cdot t}=\\
&\eta_{p,\ov{i}}(U)\cdot
\operatornamewithlimits{Res}_{t=0}\frac{st(\ov{i})(v)\cdot q(t)\cdot dt}{t^{p\cdot\ddim(u)+1}}=
RHS
\end{split}
\end{equation*}
The second equality follows from the fact that
$\gamma_{st(\ov{i})}(x)=x\cdot (\iis\cdot t^{p-1})+\ldots+x^p\,\,$, and so
$$
st(\ov{i})^{t^{-(p-1)codim(v)}}(v)=\iis^{codim(v)}\cdot pr(v).
$$
\Qed
\end{proof}

\begin{rem}
\label{sym-ln}
Similar result can be obtained with the help of Landweber-Novikov (or Steenrod) operations, but then
the number $\eta_{p,\ov{i}}(U)$ will be substituted by $p\cdot\eta_{p,\ov{i}}(U)$, and such a difference
is crucial for $p$-torsion elements. This subtlety of Symmetric operations comes from the fact that
these operations encode $p$-divisibility of certain characteristic numbers, and in reality,
all $p$-primary divisibilities of characteristic numbers are controlled by
compositions of Symmetric operations related to $p$.
\end{rem}

As an illustration we have:

\begin{cor}
\label{nunv}
Let $u\in\laz_{p^r-1}$ be a $\nu_r$-element, and $v\in\Omega^m(X)$, where $m<\frac{p(p^r-1)}{p-1}$.
Then
$$
u\cdot v=0\,\,\Rightarrow\,\, pr(v)=0\in\op{CH}^*(X)\otimes\zz_ {(p)}.
$$
\end{cor}

\begin{proof}
It follows from Proposition \ref{uv} that $\eta_{p,\ov{i}}(u)\cdot pr(v)=0\in\op{CH}^*(X)[\iis^{-1}]$,
where $\eta_{p,\ov{i}}(u)\in\zz[\iis^{-1}]$ is relatively prime to $p$ by Proposition \ref{il}.
\Qed
\end{proof}

The above result can be used, for example, to compute the Algebraic Cobordism theory and Chow groups of
a Rost motive and a Pfister quadric - see \cite[the proof of Theorem 4.1]{so2}.
The same methods give the computation of
the Algebraic Cobordism of a {\it generalized Rost motive} (an analogue of the Rost motive for $p>2$).

The operation $\Phi=\Phi(\ov{i})$ is not additive, but is very close to such:

\begin{prop}
\label{addPhi}
$$
\Phi(u+v)=\Phi(u)+\Phi(v)+f_p(u,v),
$$
where $f_p(u,v)=\sum_{l=1}^{p-1}\frac{\binom{p}{l}}{p}u^lv^{p-l}$ is a polynomial of degree $p$ in $u,v$.
\end{prop}

\begin{proof}
For cellular spaces, where the division by $[p]_{\Omega}=\frac{p\cdot_{\Omega}t}{t}$ is well-defined, the statement
follows directly from the definition of $\Phi$ and the fact that $St$ is additive.
The general case follows from Theorem \ref{MAIN} considering the external
($u\in\Omega^*(X),v\in\Omega^*(Y)$) version of the statement.
\Qed
\end{proof}

The following statement describes the multiplicative properties of the Total Symmetric operation.

\begin{prop}
\label{multPhi}
Let $u,v\in\Omega^*(X)$. Let $St=St(\ov{i})$, $\Phi=\Phi(\ov{i})$, and
$[p]_{\Omega}=\frac{p\cdot_{\Omega}t}{t}$. Then
$$
\Phi(u\cdot v)=\big(\Phi(u)\cdot St(v)+St(u)\cdot\Phi(v)+\Phi(u)\Phi(v)\cdot [p]_{\Omega}\big)_{\leq 0}
$$
\end{prop}

\begin{proof}
Considering the external version of this statement ($u\in\Omega^*(X), v\in\Omega^*(Y)$),
and fixing $u$ (respectively $v$), we obtain from Theorem \ref{MAIN}
that it is sufficient to check our statement for $(\pp^{\infty})^{\times l}$, for all $l$.
We know that $St(u)=\delta(u)+\square^p(u)-\Phi(u)\cdot [p]_{\Omega}$,
for some $\delta(u)\in\Omega^*(X)[[t]]t$, and similar for $v$.
Using the multiplicative property of $St$ we get:
$$
\square^p(u\cdot v)-St(u\cdot v)=
\big(\Phi(u)\cdot St(v)+St(u)\cdot\Phi(v)+\Phi(u)\cdot\Phi(v)\cdot [p]_{\Omega}\big)
\cdot [p]_{\Omega}+\delta',
$$
where $\delta'\in\Omega^*(X)[[t]]t$.
Since now $X$ is cellular, the division by $[p]_{\Omega}$ is unique, and we get:
$$
\Phi(u\cdot v)=\big(\Phi(u)\cdot St(v)+St(u)\cdot\Phi(v)+\Phi(u)\Phi(v)\cdot [p]_{\Omega}\big)_{\leq 0}.
$$
\Qed
\end{proof}

The action of the Symmetric operations simplifies substantially in the case of
the {\it Graded Algebraic Cobordism}.
Recall, that $Gr\Omega^*=\oplus_{r\geq 0}\Omega^*_{(r)}$, where
$\Omega^*_{(r)}(X)=F^r\Omega^*(X)/F^{r+1}\Omega^*(X)$, and $F^r\Omega^*(X)$ consists on elements supported
on closed subschemes of codimension $\geq r$ (that is, vanishing on open complements to such
subschemes).
Any cohomological operation preserves the
support of the element, and so acts on the graded theory as well.
We have natural surjection of $\laz$-modules:
$\op{CH}^r(X)\otimes_{\zz}\laz\twoheadrightarrow\Omega^*_{(r)}(X)$, which we simply denote as
$z\otimes u\mapsto z\cdot u$.
The action of Steenrod and Symmetric operations in the graded case can be described as follows
(as above, we drop $(\ov{i})$ from the notations):

\begin{prop}
\label{grad}
Let $z\in\op{CH}^r(X)$, $u\in\laz_d$. Then:
\begin{itemize}
\item[$1)$ ] $St(z\cdot u)=z\cdot t^{r(p-1)}\cdot\iis^r\cdot St(u)$;
\item[$2)$ ] $\Phi(z\cdot u)=z\cdot t^{r(p-1)}\cdot\iis^r\cdot
\Phi(u)_{\leq -r(p-1)}$, where $r\neq 0$.
\end{itemize}
\end{prop}

\begin{proof}
The case $r=0$ is trivial. For $r>0$,
since we are working modulo elements supported in codimension $\geq (r+1)$, we can assume that
$z$ is represented by a regular embedding, which gives that $St(z)=z\cdot\iis^r\cdot t^{r(p-1)}$.
Then part 1) follows from the multiplicativity of $St$.

The second part follows from Proposition \ref{multPhi} taking into account that $\Phi(z)=0$
(see Proposition \ref{emb}).
\Qed
\end{proof}

Using the action of Symmetric operations on $Gr\Omega^*$ we prove in \cite[Therem 4.3]{ACMLR}
that the Algebraic Cobordism $\Omega^*(X)$ as a module over $\laz$ has {\it relations}
in positive codimensions (and the same holds for the graded version).
This extends the result of M.Levine-F.Morel \cite[Theorem 4.4.7]{LM}
claiming that the {\it generators} of this $\laz$-module are in non-negative codimensions.

This, in particular, gives the computation of the Algebraic Cobordism of a curve:
$$
\Omega^*(C)=(\laz\otimes_{\zz}\op{Pic}(C))\oplus\laz\cdot 1_C.
$$

\end{document}